\documentclass[12pt, reqno]{amsart}

\usepackage{graphicx}
\usepackage{amsmath}
\usepackage{bbm}
\usepackage{eepic}
\usepackage{xcolor}
\selectcolormodel{gray}
\usepackage{tikz}
\usepackage{amsmath}
\usepackage{a4wide}
\usepackage[utf8]{inputenc}
\usepackage{amssymb}
\usepackage{amsopn}
\usepackage{epsfig}
\usepackage{amsfonts}
\usepackage{latexsym}
\usepackage{amsthm}
\usepackage{enumerate}
\usepackage[UKenglish]{babel}
\usepackage{comment}
\usepackage{verbatim}
\usepackage{color}
 \usepackage{mathrsfs}
\usepackage{subfig}
\usepackage{graphicx}
\usepackage{float}

\newtheorem{thm}{Theorem}[section]
\newtheorem{lma}[thm]{Lemma}

\newtheorem{prop}[thm]{Proposition}

\newcommand{\N}{\mathbb{N}}

\newcommand{\E}{\mathbb{E}}
\renewcommand{\P}{\mathbb{P}}
\newcommand{\I}{\mathcal{I}}
\newcommand{\W}{\mathscr{W}}

\renewcommand{\i}{\mathtt{i}}
\renewcommand{\j}{\mathtt{j}}
\renewcommand{\k}{\mathtt{k}}

\begin{document}

\title[]{How long is the Chaos Game?}
\author{Ian D. Morris} \address{Department of Mathematics, University of Surrey, Guildford, GU2 7XH}
\email{i.morris@surrey.ac.uk}

\author{Natalia Jurga} \address{Mathematical Institute, University of St Andrews, Scotland, KY16 9SS}
\email{naj1@st-andrews.ac.uk}

\dedicatory{Ian Morris thanks his former teacher, John Little, for introducing him to fractals in general and the chaos game in particular at New College, Swindon, in the 1995-96 academic year.}
 \maketitle
 \begin{abstract}
 In the 1988 textbook \emph{Fractals Everywhere}  M. Barnsley introduced an algorithm for generating fractals through a random procedure which he called the \emph{chaos game}. Using ideas from the classical theory of covering times of Markov chains we prove an asymptotic formula for the expected time taken by this procedure to generate a $\delta$-dense subset of a given self-similar fractal satisfying the open set condition.\\\\
 
 MSC2010 Primary: 28A80; Secondary 00A08, 60J10
 \end{abstract}

 \section{Introduction}
 
An \emph{iterated function system} or \emph{IFS} is defined to be a tuple of contracting transformations of a complete metric space, which in this article will be taken to be $\mathbb{R}^d$. It is well-known that if $(S_1, \ldots, S_N)$ is such an IFS then then there exists  a unique non-empty, compact set $F= \bigcup_{i=1}^N S_iF \subseteq \mathbb{R}^d$ which is called the \emph{attractor} or \emph{limit set} of $(S_1,\ldots,S_N)$. The properties of attractors of iterated function systems and the natural measures supported on them have been the subject of substantial mathematical inquiry for several decades since their introduction in  \cite{BaDe85, Hu81}, and remain a highly active topic of contemporary mathematical research (we note for example \cite{BaHoRa19,DaSi17,Ho14,LiVa16,Sh19,SiWe19,Wu19}) as well as being noted for their aesthetic appeal. 

In \cite{Ba88} Barnsley introduced an algorithm known as the \emph{Chaos Game} for the construction of the limit set $F$ of an iterated function system $(S_1,\ldots,S_N)$. Given an arbitrary starting point $x_0 \in \mathbb{R}^d$, we define a sequence $(x_n)_{n=0}^\infty$ inductively by choosing for each $n \geq 1$ an index $i_n \in \{1,\ldots,N\}$ independently at random according to some fixed non-degenerate probability vector $(p_1,\ldots,p_N)$, and taking $x_n:=S_{i_n}x_{n-1}$ for every $n \geq 1$. It is not difficult to show that the resulting sequence almost surely has the attractor $F$ as its $\omega$-limit set (that is, we have $\bigcap_{m=1}^\infty \overline{\{x_n \colon n \geq m\}}=F$) and it is not much more difficult to show that the distribution $\frac{1}{n}\sum_{k=0}^{n-1}\delta_{x_k}$ converges almost surely to the unique Borel probability measure $m$ supported on $F$ which satisfies $m=\sum_{i=1}^N p_i (S_i)_*m$, as was first established in \cite{El87}. If the initial point $x_0$ is taken to be in the attractor (for example, by taking $x_0$ to be the fixed point of one of the contractions $S_i$) then one obtains the simpler result that the sequence $(x_n)_{n=0}^\infty$ is almost surely dense in the attractor, and for the rest of this article we will prefer to make this assumption on the starting point $x_0$. Yet surprisingly, we have found no trace in the literature of the following question: \emph{how quickly does the randomly-generated sequence $(x_n)$ become dense in the attractor?} In this direction we are aware only of the article \cite{GuIgRo96}, which informally investigates the problem of choosing probabilities in such a way as to generate fractal images with maximal efficiency using the chaos game procedure. In the present note we attempt to fill this gap in the literature with a rigorous investigation.

Let us make our question precise.  Given a compact subset $F$ of $\mathbb{R}^d$ we will say that a subset $X$ of $F$ is \emph{$\delta$-dense} in $F$ if for every $z \in F$ there exists $x \in X$ such that $d(x,z)\leq \delta$ in the standard metric on $\mathbb{R}^d$. (Since $X$ is a subset of $F$, this is equivalent to asking that the Hausdorff distance between $\overline{X}$ and $F$ is at most $\delta$.) Given an IFS $(S_1,\ldots,S_N)$ with attractor $F$, for each $\delta>0$, $\i=i_1i_2 \ldots \in \{1,\ldots,N\}^{\mathbb{N}}$ and starting point $v \in F$ we define the $\delta$-\emph{waiting time} along the sequence $\i$ as
\[W_{\delta,v}(\i):=\inf\left\{n \geq 1 \colon \left\{S_{i_1}v,S_{i_2}S_{i_1}v,\ldots,S_{i_n}\cdots S_{i_1}v\right\} \text{ is }\delta\text{-dense in $F$}\right\}.\]
If additionally a nondegenerate probability vector $(p_1,\ldots,p_N)$ is understood, then we define the \emph{expected $\delta$-waiting time} with starting point $v$ to be the expectation $\mathbb{E}(W_{\delta,v})$ with respect to the $(p_1,\ldots,p_N)$-Bernoulli measure on $\{1,\ldots,N\}^{\mathbb{N}}$.

We recall that an IFS is said to satisfy the \emph{open set condition} or \emph{OSC} if there exists a nonempty open set $U\subset \mathbb{R}^d$ such that $\bigcup_{i=1}^N S_iU \subseteq U$ with the sets $S_iU$ being pairwise disjoint.  The set $U$ may without loss of generality be taken to be bounded, and we will always assume that this is the case. We recall that $S_i \colon \mathbb{R}^d \to \mathbb{R}^d$ is called a \emph{similitude} or \emph{similarity transformation} if there exists $r_i \in (0,1)$ such that $d(S_iu,S_iv) =r_i d(u,v)$ for all $u,v \in \mathbb{R}^d$; in this case we say that $r_i$ is the \emph{contraction ratio} of $S_i$. If $(S_1,\ldots,S_N)$ is an IFS of similitudes with respective contraction ratios $r_1,\ldots,r_N$ then the \emph{similarity dimension} of $(S_1,\ldots,S_N)$ is defined to be the unique real number $s \geq 0$ such that $\sum_{i=1}^N r_i^s=1$. By a classical theorem of Hutchinson (see \cite{Hu81}), if an IFS of similitudes satisfies the open set condition then the Hausdorff and box dimensions of the attractor $F$ are both equal to the similarity dimension $s$. It was also shown by Hutchinson that there exists a unique Borel probability measure $m$ satisfying $m=\sum_{i=1}^N r_i^s (S_i)_*m$, and that this measure is supported on $F$ and has Hausdorff dimension equal to that of $F$; moreover, if any other probability vector is chosen then the resulting measure has Hausdorff dimension smaller than that of $F$. At an intuitive level this suggests that the limit distribution of the random sequence $(x_n)_{n=0}^\infty$ generated by the Chaos Game will be most evenly distributed around the attractor when the underlying probability vector is $(r_1^s,\ldots,r_N^s)$, and will be more concentrated in  certain subregions of the attractor for other probability vectors. Thus we might expect the probability vector $(r_1^s,\ldots,r_N^s)$ to generate a random sequence which fills up the attractor most efficiently, and for other choices of probability vector to result in longer waiting times for the sequence to become $\delta$-dense in the attractor. This intuition is realised in our main result:

\begin{thm}\label{natural}
Let $(S_1, \ldots, S_N)$ be an IFS of similitudes $S_i \colon \mathbb{R}^d \to \mathbb{R}^d$, with contraction ratios given by $r_i \in (0,1)$, which satisfies the OSC. Let $s$ denote the similarity dimension of $(S_1,\ldots,S_N)$, let $(p_1,\ldots,p_N)$ be a nondegenerate probability vector, and define
\begin{eqnarray}
t:= \max_{1 \leq i\leq N} \frac{\log p_i}{\log r_i},\label{t thing}
\end{eqnarray}
where we observe that by the arithmetic-geometric mean inequality we have $t\geq s$ with equality if and only if $(p_1,\ldots,p_N)=(r_1^s,\ldots,r_N^s)$. If the maximum in (\ref{t thing}) is attained at a unique value of $i \in \{1, \ldots, N\}$ then there exists a constant $C>0$ such that for every starting point $v \in F$ and $0<\delta< \min_{1 \leq i \leq N} r_i$,
\begin{eqnarray}
C^{-1} \delta^{-t}\log\log\left(\frac{1}{\delta}\right) \leq \mathbb{E}(W_{\delta,v}) \leq C \delta^{-t} \left(\log\log\left(\frac{1}{\delta}\right)\right)^2.\label{other eqn1}
\end{eqnarray}
If the maximum in (\ref{t thing}) is \emph{not} attained at a unique value of $i \in \{1, \ldots, N\}$ then there exists a constant $C>0$ such that for every starting point $v \in F$ and $0<\delta< \min_{1 \leq i \leq N} r_i$, 
\begin{eqnarray}
C^{-1} \delta^{-t}\log \left(\frac{1}{\delta}\right) \leq \mathbb{E}(W_{\delta,v}) \leq C \delta^{-t} \log \left(\frac{1}{\delta}\right).\label{other eqn2}
\end{eqnarray}
\end{thm}
Thus for the probability vector $(p_1,\ldots,p_N)=(r_1^s,\ldots,r_N^s)$ we have for every starting point $v \in F$
\[C^{-1} \delta^{-s}\log \left(\frac{1}{\delta}\right) \leq \mathbb{E}(W_{\delta,v}) \leq C \delta^{-s} \log \left(\frac{1}{\delta}\right)\]
for all $0<\delta< \min_{1 \leq i \leq N} r_i$, and for every other probability vector $\mathbb{E}(W_{\delta,v})$ tends to infinity more rapidly as $\delta \to 0$. At an intuitive level the principle underlying this result is that for the ``natural'' probability measure $(r_1^s,\ldots,r_N^s)$, all regions of the attractor with diameter $\delta$ take an approximately equal time to visit; for other probability measures, some $\delta$-balls in the attractor are substantially more difficult to access than others. As we will see below, the key determiner of the expected waiting time is the expected time taken to visit the most slowly accessible part of the attractor, and it transpires that this in turn corresponds to a region of the form $S_i^nU$ where $i$ is chosen to maximise the ratio $\log p_i / \log r_i$ and $n \geq 1$ is chosen such that this region has diameter approximately $\delta$. 
In the case where $\log p_i / \log r_i$ is maximised at a unique index $i$ it is interesting to ask to what extent the result \eqref{other eqn1} may be sharpened, but we have not been able to determine the exact rate of growth of $\mathbb{E}(W_{\delta,v})$ in that case in the present article. It is also interesting to ask what information may be obtained regarding the pointwise almost sure behaviour of the family of random variables $W_{\delta,v}$ for fixed $v$.

In the case where the starting point $v$ is not taken to be in the attractor, since the sequence $S_{i_n}\cdots S_{i_1}v$ approaches the attractor at a uniform exponential rate, one may obtain the same asymptotics for the expected waiting time as in Theorem \ref{natural} but with a larger constant $C$ depending on the initial distance between $v$ and the attractor; we leave the details of this adaptation of our result to the reader.

\begin{figure}
    \centering
    \subfloat[Probability vector $(p_1,p_2,p_3):=(\frac{1}{3},\frac{1}{3},\frac{1}{3})$. The observed $\delta$-waiting time was $W=3408$. The expected value for the $\delta$-waiting time $W$ according to \eqref{other eqn2} was of the order of magnitude of $\delta^{-t}\log(\frac{1}{\delta}) \simeq 3032$.]{{\includegraphics[width=5.8cm]{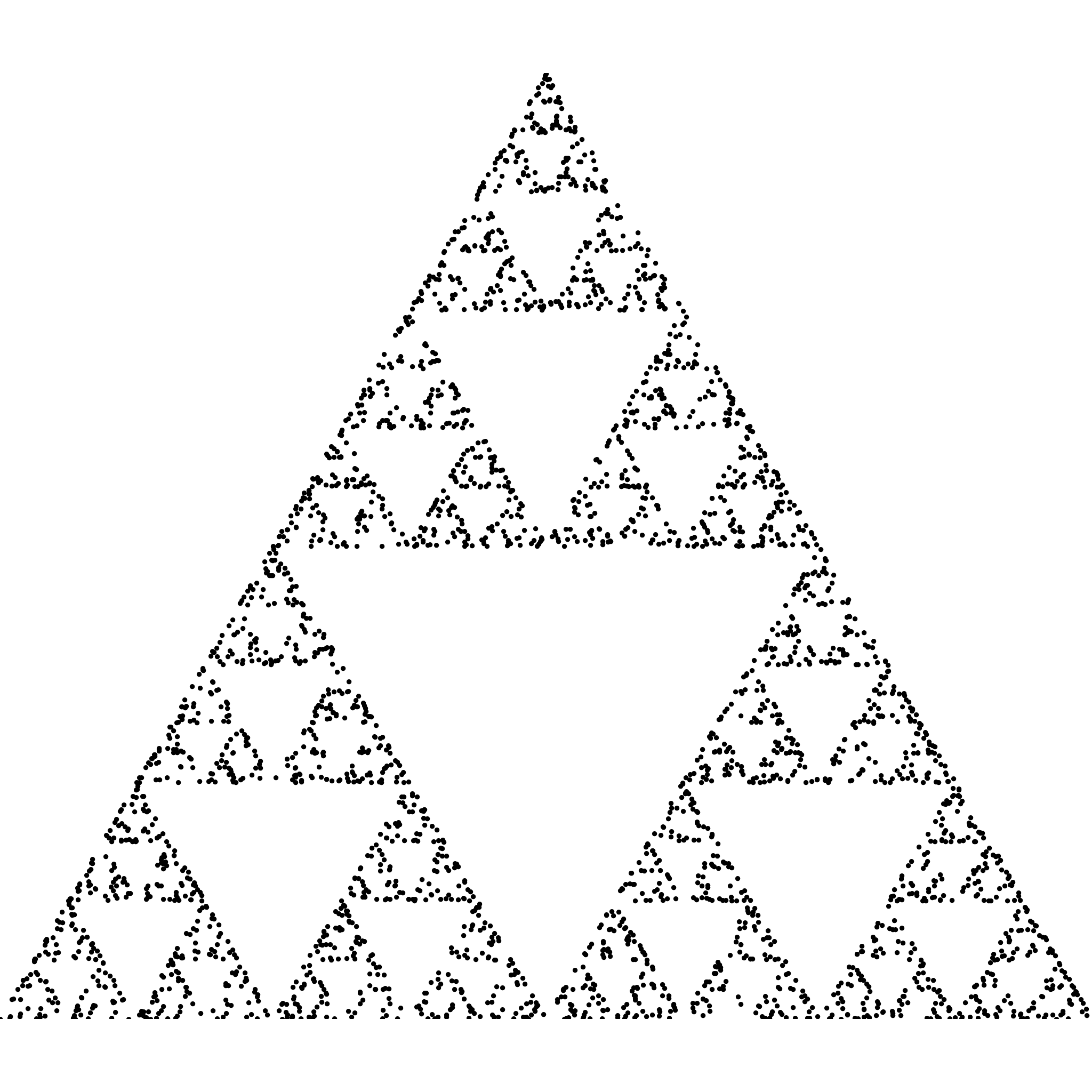}}}
    \qquad
    \subfloat[Probability vector $(p_1,p_2,p_3):=(\frac{1}{4},\frac{1}{4},\frac{1}{2})$. The observed $\delta$-waiting time was $W=18732$.  The expected value for the $\delta$-waiting time $W$ according to \eqref{other eqn2} was of the order of magnitude  of $\delta^{-t}\log(\frac{1}{\delta}) \simeq17035$.]{{\includegraphics[width=5.8cm]{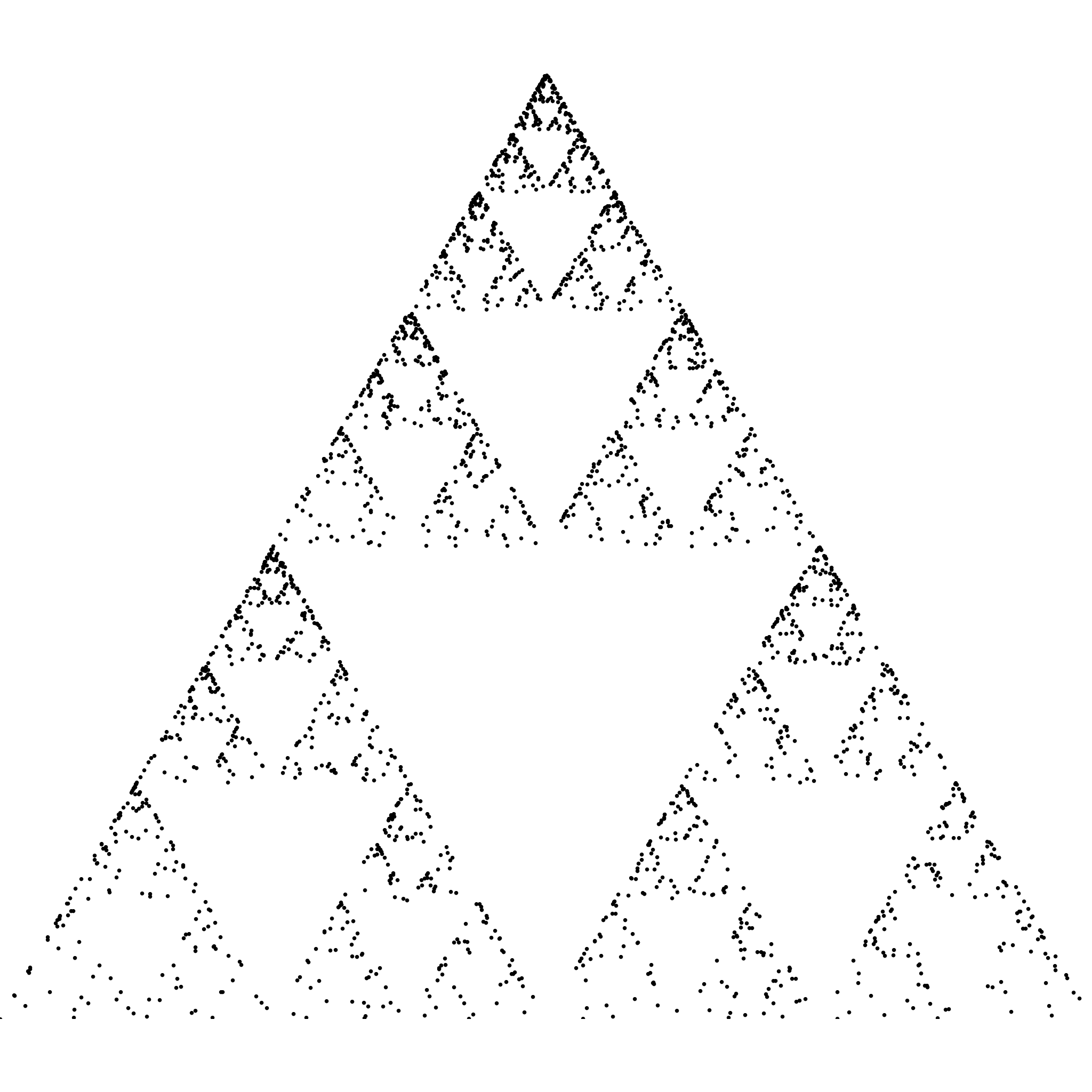}}}
    \caption{Each of these diagrams shows a randomly-generated sequence $(x_n)_{n=0}^W$ starting at $x_0:=0$ and terminating when every point of the Sierpinski triangle has been approached to within distance $\delta:=2^{-6}$. Here the underlying iterated function system is given by $T_1(x):=\frac{x}{2}$, $T_2(x):=\frac{x}{2} + \frac{1}{2}$ and $T_3(x):=\frac{x}{2} + \frac{\sqrt{3}}{4}$. As indicated by Theorem \ref{natural} and suggested by the accompanying heuristic description, with unbalanced probabilities the trajectory spends much more time confined to a small region closest to the fixed point of the transformation carrying highest probability weight, and requires correspondingly more points in order to terminate.}
    \label{fi:onlyfigure}%
\end{figure}

To illustrate the structure of the proof of Theorem \ref{natural} it is helpful to consider a simpler case in which the contraction ratios $r_i$ are all equal to the same constant $r$ and the diameter of the set $U$ is precisely $1$. In this case, if every set of the form $S_{i_m}\cdots S_{i_1}U$ has been visited by the sequence $(x_n)_{n=0}^\infty$ by time $N$ then we certainly have $W_{r^m,v} \leq N$. On the other hand one may shown that there exists $\kappa>0$ such that for every $m \geq 1$, every set of the form $S_{i_m}\cdots S_{i_1}U$ contains an open ball of radius $r^m\kappa$ which in particular does not intersect any \emph{other} set of the form $S_{i_m}\cdots S_{i_1}U$. Thus if the sequence of indices $\i$ fails to include a particular string of the form $i_1\cdots i_m$ before time $N$, we expect that $W_{r^m\kappa,v}> N$. (There is some imprecision here in that the initial point $v$ may by chance have belonged to the ball $r^m\kappa$, but it transpires that this imprecision has a negligible effect in practice.) This suggests that the asymptotic behaviour of the expectation of $W_{\delta,v}$ can be reduced to the problem of determining the expected first time for an IID random sequence in $\{1,\ldots,N\}^{\mathbb{N}}$, chosen with respect to the Bernoulli measure $(p_1,\ldots,p_N)^{\mathbb{N}}$, to include all of the distinct words of length $m$ over the alphabet $\{1,\ldots,N\}$, where $m$ is chosen so that $r^m$ is approximately the size of $\delta$. But this symbolic problem is precisely the classical \emph{coupon collector's problem} described in, for example, \cite{LePe17}.  In the full generality of Theorem \ref{natural} this approach must be adapted somewhat: since the sets $S_{i_m}\cdots S_{i_1}U$ will in general have different diameters for different sequences $i_1,\ldots,i_m$ of the same length $m$, it is necessary to partition the set $\{1,\ldots,N\}^{\mathbb{N}}$ into cylinders and estimate the expected time for all of these cylinders to be visited by a random sequence. This results in a Markov chain analogue of the coupon collector's problem which we solve using techniques adapted from \cite[\S 11]{LePe17}.

Our proof of Theorem \ref{natural} will thus be divided into two parts: the reduction of the problem to a covering problem for Markov chains, and the solution of the latter covering problem. Some possible directions for future research are described at the end of this note.

\section{A Markov chain construction}

For the remainder of this article we fix an IFS of similarities $(S_1, \ldots, S_N)$ which satisfies the OSC and denote the contraction ratio of each map $S_i$ by $r_i$. We also fix a nondegenerate probability vector $(p_1, \ldots, p_N)$. Let us write $r_{\min}:=\min_{i=1, \ldots, N} r_i$. In this section we will show that for each $0<\delta< r_{\min}$, we can construct a Markov chain whose expected time to visit all of its states is approximately proportional to $\E W_{\delta, v_0}$. More precisely, given a Markov chain $(X_n)_{n=0}^\infty$ on a finite state space $\Omega$, we define the \emph{covering time} by
$$\tau_{\textnormal{cov}}=\min\{t \geq 0: \, \forall y \in \Omega, \; \exists s \leq t \; \textnormal{s.t.} \; X_s=y\},$$
that is, the first time that all of the states in $\Omega$ have been visited by the Markov chain. Given $x \in \Omega$, we denote by $\E_x \tau_{\textnormal{cov}}$ the expected covering time given that $X_0=x$. We can now state the main result of this section.

\begin{prop}\label{main2}
For each $0<\delta<r_{\min}$ there exists an irreducible Markov chain $(X_n^\delta)_{n=0}^\infty$ on a finite state space $P_\delta$ such that for each $v_0 \in F$, there exists $\i_0 \in P_\delta$ for which
\begin{eqnarray}
\E W_{\frac{\delta}{c},v_0} \leq \E_{\i_0} \tau_{\textnormal{cov}}  \leq  \E W_{c\delta,v_0} \label{main2-eqn}
\end{eqnarray}
where the constant $c\in (0,1)$ is independent of $\delta$, $v_0$ and $\i_0$.
\end{prop}
We will derive Proposition \ref{main2} from the combination of two results to be proved below, Proposition \ref{equiv} and Proposition \ref{chain}. The significance of (\ref{main2-eqn}) is that for each $0<\delta<r_{\min}$ and $\i_0 \in P_\delta$, $\E_{\i_0} \tau_{\textnormal{cov}}$ can be estimated by employing classical methods for bounding covering times of irreducible Markov chains. In particular, we will show that for any $0<\delta<r_{\min}$ and $\i_0 \in P_\delta$, $\E_{\i_0} \tau_{\textnormal{cov}}$ satisfies the upper and lower bounds presented in (\ref{other eqn1}) and (\ref{other eqn2}), which  will directly imply the bounds for $\E W_{\delta, v_0}$ by (\ref{main2-eqn}), up to a change in the uniform constant.


We introduce some notation. Define $\I=\{1, \ldots, N\}$, which we call the \emph{index set}. Let $\I^n=\{i_1 \ldots i_n: i_j \in \I\}$ denote the set of all words of length $n$ over the index set, $\I^*=\bigcup_{n \in \N} \I^n$ the set of all finite words over the index set and $\Sigma=\I^{\N}$ the set of all sequences over the index set. $\Sigma$ is equipped with the infinite product topology with respect to which it is compact and metrisable. If $\i \in \I^*$ and $\j \in \I^* \cup \Sigma$ we let $\i\j$ denote the concatenation of $\i$ with $\j$. Given $\i \in \I^*$, let $[\i]$ denote the cylinder set $[\i]=\{\i\j: \j \in \Sigma\}$. Cylinder sets are clopen and generate the topology on $\Sigma$.  Given $\i=i_1 \ldots i_n \in \I^*$ let $|\i|$ denote the length of the word $\i$, so that in this case $|\i|:=n$. Given $\i=i_1 \ldots i_n \in \I^*$ with $|\i| \geq 2$ denote $\i^-:= i_1 \ldots i_{n-1}$.  Given $\i=i_1i_2 \ldots \in \Sigma$ or $\i=i_1 \ldots i_{n+m} \in \I^*$ let $\i|_n:= i_1 \ldots i_n$. Given $\i =i_1 \ldots i_n\in \I^*$ let $r_\i:=r_{i_1} \cdots r_{i_n}$, $p_{\i}:=p_{i_1} \cdots p_{i_n}$ and $S_\i:= S_{i_1}\cdots S_{i_n}$.

It is not difficult to show that for every $\i \in \Sigma$ the limit
\[\pi(\i):=\lim_{n \to \infty} S_{i_1}\cdots S_{i_n}v\]
exists for every $v \in \mathbb{R}^d$ and that moreover the limit is independent of the starting point $v$. This \emph{coding map} $\pi \colon \Sigma \to \mathbb{R}^d$ is continuous and its image $\pi(\Sigma)$ is precisely $F$. (However, in cases where $\overline{T_iU}\cap \overline{T_jU} \neq \emptyset$ for some $i \neq j$ the coding map $\pi$ can fail to be injective.)
Obviously, for every $\i \in \Sigma$ and $n \geq 1$ we have $\i \in [\j]$ where $\j$ is the word corresponding to the first $n$ symbols of $\i$. Since $\pi \colon \Sigma \to F$ is surjective, for every $x \in F$ and $n \geq 1$ there exists at least one word $\j \in \I^*$ with $|\j|=n$ such that $x \in \pi([\j])=S_{\j}F$.

For each $\delta \in (0,r_{\min})$, define a subset of $\I^*$ by
$$P_{\delta}=\{\i \in \I^*: r_{\i} \leq \delta < r_{\i^-}\}.$$
Note that since $\delta< r_{\min}$, if $r_{\i} \leq \delta$ then necessarily $|\i| \geq 2$ and hence $\i^-$ is well defined. It is easy to see that $\{[\i]: \i \in P_{\delta}\}$ is a finite partition of $\Sigma$. Therefore, for any $x \in F$ there exists at least one word $\i \in P_\delta$ such that $x \in S_\i F$.  Let $N_{\delta}=|P_{\delta}|$, where $|P_{\delta}|$ denotes the cardinality of the set $P_{\delta}$. We claim that $\delta^{-s} \leq N_{\delta} \leq r_{\min}^{-s}\delta^{-s}$. To see this let $\mathbb{P}$ denote the Bernoulli probability measure on $\Sigma$ defined by $\mathbb{P}([\i])=r^s_\i$ for every $\i \in \I^*$. Then $\sum_{\i \in P_\delta} \mathbb{P}([\i])=1$ since $\{[\i]\colon \i \in P_\delta\}$ is a partition, which is to say $\sum_{\i \in P_\delta}r_\i^s=1$. Therefore $r_{\min}^s\delta^sN_\delta < 1 \leq \delta^sN_\delta$ which proves the claim.


We say that a list of words $\j_1,\ldots,\j_k \in \I^*$ \emph{visits} the cylinder set $[\i]\subset \Sigma$ if at least one of the words $\j_i$ satisfies $[\j_i] \subseteq [\i]$. We will show that instead of keeping track of which regions of the attractor are visited by the chaos game algorithm, we can keep track of which cylinder sets in $\{[\i]: \i \in P_\delta\}$ are visited by a symbolic analogue of the algorithm. This is made precise in the following proposition.

\begin{prop} \label{equiv}
There exist $\kappa>0$ and $\Delta>0$ depending only on $(S_1,\ldots,S_N)$ having the following property. Let $v_0 \in F$ be arbitrary, $\delta \in (0, r_{\min})$ and choose any $\i_0 \in P_\delta$ such that $v_0 \in S_{\i_0} F$:
\begin{enumerate}[(i)]
\item
If $\i_0, i_1\i_0, \ldots, i_n\ldots i_1 \i_0$ visits every cylinder in $\{[\i]:\i \in P_{\delta}\}$ then the set \\ $\{v_0, S_{i_1}v_0, \ldots, S_{i_n} \cdots S_{i_1} v_0\}$ is $2\delta\Delta$-dense in $F$.
\item
If $\{v_0, S_{i_1}v_0, \ldots, S_{i_n} \cdots S_{i_1} v_0\}$ is $\kappa\delta$-dense in $F$ then $\i_0, i_1\i_0, \ldots, i_n\ldots i_1 \i_0$ visits each cylinder in $\{[\i]: \i \in P_{\delta}\}$.
\end{enumerate}
\end{prop}
 
\begin{proof}
Since the IFS $(S_1, \ldots, S_N)$ satisfies the OSC, by a result of A. Schief (\cite{Sc94}) it also satisfies the \emph{strong open set condition} (SOSC), that is, there exists a bounded open set $U$ such that $\bigcup_{i=1}^N S_iU \subset U$ where the union is disjoint and $U \cap F \neq \emptyset$. Let $\Delta:= \textnormal{diam} F$. It follows that there exists $x \in F$ and $0<\epsilon< \Delta$ with $B(x,\epsilon) \subset U$. In particular for $\i, \j \in P_{\delta}$ with $\i \neq \j$, the balls $S_{\i}B(x, \epsilon)$ and $S_{\j}B(x, \epsilon)$ are disjoint. If $y \in F$ is arbitrary then we may choose $\i \in P_\delta$ such that $y \in S_\i F$. Since $S_{\i}x, y \in S_{\i} F$ it then follows that $|S_{\i}x-y| \leq r_{\i} \Delta$. We have shown that
\begin{eqnarray}F \subset \bigcup_{\i \in P_{\delta}} B( S_{\i}x,r_{\i}\Delta) . \label{balls} \end{eqnarray}

Now, to prove (i), fix $\i_0, i_1\i_0, \ldots, i_n\ldots i_1 \i_0$ satisfying the hypothesis of (i) and consider an arbitrary $\i \in P_{\delta}$. By assumption there exists $k$ satisfying $0 \leq k \leq n$ such that $[i_k \ldots i_1 \i_0] \subset [\i]$. Thus $S_{i_k} \cdots S_{i_1}v_0 \in S_\i F$, and in particular $S_{i_k} \cdots S_{i_1}v_0 \in B(S_\i x, r_\i \Delta)$. Moreover, since for any $y \in B(S_\i x, r_\i \Delta)$ we have
$$|y- S_{i_k} \cdots S_{i_1} v_0| \leq |y-S_\i x|+|S_\i x-S_{i_k} \cdots S_{i_k} \cdots S_{i_1} v_0| <r_\i \Delta +r_\i \Delta \leq 2\delta \Delta,$$
it follows that $B(S_{\i}x, r_\i \Delta) \subset B(S_{i_k} \cdots S_{i_1}v_0, 2\delta\Delta)$. Since $\i \in P_{\delta}$ was arbitrary, it follows that
$$F \subset \bigcup_{\i \in P_{\delta}} B( S_{\i}x,r_{\i}\Delta) \subset  \bigcup_{k=0}^n B(S_{i_k} \cdots S_{i_1}v_0, 2\delta\Delta),$$
in other words, $v_0, S_{i_1}v_0, \ldots, S_{i_n} \cdots S_{i_1} v_0$ is $2\delta\Delta$-dense in $F$ as required to establish (i).

To prove (ii), define $\kappa:= r_{\min} \epsilon$. For each $\i \in P_{\delta}$ we have $S_{\i}x \in S_{\i}F \subset F$  and $S_{\i}(B(x,\epsilon))=B(S_{\i}x,r_\i \epsilon)$, and we note also that  $\delta \kappa = \delta r_{\min}\epsilon \leq r_{\i}\epsilon$ where we have used the definition of $P_\delta$. Thus if $\{v_0, S_{i_1}v_0, \ldots, S_{i_n} \cdots S_{i_1} v_0\}$ is $\kappa\delta$-dense in $F$ then for each $\i \in P_{\delta} \setminus \{\i_0\}$ there must exist $k$ satisfying $0 \leq k \leq n$ such that $S_{i_k} \cdots S_{i_1}v_0$ is $\kappa\delta$-close to $S_{\i}x$ and therefore satisfies $S_{i_k} \cdots S_{i_1}v_0 \in B(S_{\i}x,\kappa \delta) \subseteq B(S_{\i}x,r_\i \epsilon)=S_{\i}B(x,\epsilon)$. By the definition of $\epsilon$ we know that  $\{S_\i B(x,\epsilon)\}_{\i \in P_\delta}$ are disjoint balls, so we necessarily have $[i_k \ldots i_1 \i_0 ] \subset [\i]$ and therefore $\i_0, i_1\i_0, \ldots, i_n\ldots i_1 \i_0$ visits  every cylinder $[\i]$ such that $\i \in P_\delta$. This completes the proof.
\end{proof}


Proposition \ref{equiv} is key to the construction of the Markov chain $(X_n^{\delta})_{n=0}^\infty$, which we are now ready to provide details of:
\begin{prop} \label{chain} Let $\delta \in (0,r_{\min})$. Define a square matrix $A_\delta=[a_{\i,\j}]_{\i,\j \in P_\delta}$ of dimension $N_{\delta}$ by
\[a_{\i,\j}:=\left\{\begin{array}{cl}p_i &\text{if  $[i\i] \subset [\j]$ }\\ 0&\text{otherwise,}\end{array}\right.\]
and define a vector $\pi_\delta \in \mathbb{R}^{N_\delta}$ by $\pi_\i:=p_\i$. Then:
\begin{enumerate}[(a)]
\item $A_\delta$ is a row stochastic matrix,
\item $\pi_\delta$ is a left stationary vector, i.e. $\pi_\delta A_\delta=\pi_\delta$,
\item $ A_\delta$ is irreducible.
\end{enumerate}
\end{prop}
In order to prove Proposition \ref{chain} we require two preliminary lemmas concerning $P_\delta$.

\begin{lma}
Let $\delta \in (0, r_{\min})$. The \emph{Markov partition property} holds: for every $\i = i_1\cdots i_n \in P_\delta$ there exist $\j_1,\ldots,\j_m \in P_\delta$ such that $[\i]=\bigcup_{k=1}^m [i_1 \j_k]$.
\label{markov}
\end{lma}

\begin{proof}
If $n=1$ then we have $[\i]=[i_1]=\bigcup_{\j \in P_\delta} [i_1\j]$ using the fact that $\bigcup_{\j \in P_\delta}[\j]=\Sigma$, so assume $n\geq 2$. Given $\i =i_1\cdots i_n\in P_\delta$ with $n \geq 2$, define $\i'=i_2\cdots i_n$. Since $\{[\j] \colon \j \in P_\delta\}$ is a partition there exist  $\j_1,\ldots,\j_m \in P_\delta$ such that $[\i'] \subseteq \bigcup_{k=1}^m [\j_k]$ and $[\i']\cap [\j_k]\neq \emptyset$ for all $k=1,\ldots,m$. We appeal to the fact that if two cylinder sets intersect then one of them contains the other. If for some $k$ the set $[\i']$ is a subset of $[\j_k]$ then either $\i'=\j_k$ or $\i'=\j_k\k$ for some finite word $\k$. In the former case we have $[\i]=[i_1\i']=[i_1\j_k]$ and the proof is complete. In the latter case we have $\i = i_1\i'=i_1\j_k\k$ so that $i_1\j_k$ is a proper prefix of $\i \in P_\delta$. This implies $r_{i_1}r_{\j_k}>\delta$ and in particular $r_{\j_k}>r_{i_1}^{-1}\delta>\delta$ contradicting that $\j_k \in P_\delta$. We conclude that $[\j_k]\subseteq [\i']$ for each $k=1,\ldots m$, so $[\i'] \subseteq \bigcup_{k=1}^m [\j_k] \subseteq [\i']$ and the result $[i_1\i']=\bigcup_{k=1}^m [i_1\j_k]$ follows. 
\end{proof}

\begin{lma}\label{le:handy}
Let $\delta \in (0, r_{\min})$. Then for every $\i \in P_\delta$ and $i\in \{1,\ldots,N\}$ there exists a unique $\j \in P_\delta$ such that $[i\i]\subseteq [\j]$. 
\end{lma}
\begin{proof}
Fix such an $i$ and $\i$. To demonstrate the existence of $\j$ we observe that by the partition property there must exist $\j \in P_\delta$ such that $[i\i] \cap [\j]\neq \emptyset$. Writing $[\j]=\bigcup_{k=1}^m [i\j_k]$ using the Markov partition property (Lemma \ref{markov}) we see that there exists $k$ such that $[i\i]\cap [i\j_k]\neq \emptyset$. By the partition property this is only possible if $\i=\j_k$ and we deduce that $[i\i]\subseteq [\j]$. This proves existence. To obtain uniqueness we observe that if distinct $\j_1,\j_2 \in P_\delta$ satisfy $[i\i] \subseteq [\j_1]$ and $[i\i] \subseteq [\j_2]$ then $[\j_1]\cap [\j_2]\neq \emptyset$ and the partition property is contradicted. The lemma is proved.
\end{proof}

\vspace{3mm}

\begin{proof}[Proof of Proposition \ref{chain}]
To see that $A_\delta$ is row stochastic we observe that for each $\i \in P_\delta$ and each $i=1,\ldots,N$ there exists a unique $\j \in P_\delta$ such that $[i\i]\subseteq [\j]$ by Lemma \ref{le:handy}. This implies there exists a unique $\j \in P_\delta$ such that $a_{\i,\j}=p_i$. It follows that every $p_1,\ldots,p_N$ occurs once in the row of $A_\delta$ corresponding to $\i$ and the remaining entries in that column are zero, so $A_\delta$ is row stochastic as claimed.

That $\pi_\delta$ is a stochastic vector follows directly from the fact that $P_\delta$ is a partition of $\Sigma$. To verify the equation $\pi_\delta A_\delta=\pi_\delta$, let $\j \in P_\delta$ be arbitrary and using Lemma \ref{markov} we can write $[\j]=\bigcup_{k=1}^m [i\i_k]$ where $\i_1,\ldots,\i_m \in P_\delta$ and $i$ is the first symbol of $\j$. We obtain
\[ \sum_{\i \in P_\delta} \pi_\i a_{\i,\j} =\sum_{\substack{\i \in P_\delta \\ [i\i] \subseteq [\j]}}p_\i p_i  =\sum_{k=1}^m p_{\i_k} p_i  = p_\j=\pi_\j\]
as required, where the penultimate equation follows from $[\j]=\bigcup_{k=1}^m [i\i_k]$ and where we have used the fact that if $[i\i] \subseteq [\j]$ then $\i$ is necessarily equal to some $\i_k$ by the partition property.

To show that $A_\delta$ is irreducible, it is sufficient to show that there exists $L>0$ such that $A_{\delta}^L$ is a positive matrix. We will show that this is true for $L=L_{\delta}:= \max_{\i \in P_{\delta}}\{|\i|\}$. Let $\i,\j \in P_\delta$ be arbitrary with $\j=j_1,\ldots,j_n$, say, and if $n<L$ let $j_{n+1},\ldots,j_{L} \in \{1,\ldots,N\}$ be arbitrary. Define $\k_{L+1}:=\i$. By a simple inductive application of Lemma \ref{le:handy} starting at $t=L$ and descending to $t=1$ we may choose $\k_1,\ldots,\k_L \in P_\delta$ such that for all $t=1,\ldots,L$ we have $[j_t\k_{t+1}] \subseteq [\k_{t}]$. Define $\ell_t:=\min\{L+1-t,|\k_t|\}$ for each $t=1,\ldots,L+1$. We claim that for each $t=1,\ldots,L+1$ the first $\ell_t$ symbols of $\k_t$ are $j_{t},\ldots,j_{t-1+\ell_t}$ in that order. This statement is clearly true for $t=L+1$, so let us assume its truth for some $t \in \{2,\ldots,L+1\}$ and deduce its truth for $t-1$. Since $[j_{t-1}\k_{t}] \subseteq [\k_{t-1}]$ we have $ |\k_{t-1}|\leq 1+|\k_t|$ and therefore $\ell_{t-1}=\min\{L+2-t,|\k_{t-1}|\} \leq  \min \{L+2-t,1+|\k_t|\}  = 1+\ell_t$. The relation $[j_{t-1}\k_{t}] \subseteq [\k_{t-1}]$  also implies that the first $|\k_{t-1}|$ symbols of $j_{t-1}\k_t$ are precisely the word $\k_{t-1}$. Since the first $1+\ell_t$ symbols of $j_{t-1}\k_t$ are $j_{t-1}j_t\cdots j_{t-1+\ell_t}$, the first $\ell_{t-1} \leq \min\{|\k_{t-1}|,1+\ell_t\}$  symbols of $\k_{t-1}$ must be $j_{t-1}\cdots j_{t-2+\ell_{t-1}}$. This is precisely what is required for the claim to be true in the case $t-1$. The claim follows by induction. 

Applying the claim with $t=1$ it follows that the first $\ell_1=\min\{L,|\k_1|\}=|\k_1|$ symbols of $\k_1$ are $j_1,\ldots,j_{\ell_1}$. If $\ell_1<n$ then $[\j]$ is a proper subset of $[\k_1]$ and if $\ell_1>n$ then $[\k_1]$ is a proper subset of $[\j]$, but both of these contradict the partition property of $P_\delta$ and we conclude that $\ell_1=n$ and therefore $\k_1=\j$. The relation  $[j_t\k_{t+1}] \subseteq [\k_{t}]$ for each $t$ implies that $a_{\k_{t+1},\k_{t}}>0$ for all $t=1,\ldots,L$. Since $\k_1=\j$ and $\k_{L+1}=\i$ it follows that $a_{\i,\k_L}a_{\k_L,\k_{L-1}}\cdots a_{\k_2,\j}>0$. This is precisely what is needed to show that the entry of $A_\delta^L$ in position $(\i,\j)$ is positive. Since $\i$ and $\j$ were arbitrary the proof of (c) is complete.\end{proof}

We may now prove Proposition \ref{main2}:

\begin{proof}[Proof of Proposition \ref{main2}]
Let $\kappa>0$ be as given by Proposition \ref{equiv} and for each $\delta \in (0,r_{\min})$ let $(X_n^\delta)_{n=0}^\infty$ be the irreducible Markov chain on the state space $P_\delta$ which is induced by the transition matrix $A_\delta$ defined in Proposition \ref{chain}. Given $v_0 \in F$, choose any $\i_0 \in P_\delta$ such that $v_0 \in S_{\i_0}F$. For every $\i=i_1i_2 \ldots \in \Sigma$ define 
$$\W_{\delta, \i_0}(\i):=\inf\{m \geq 1: \forall \j \in P_{\delta}\; \exists n \in \N \; \textnormal{such that} \; 1 \leq n \leq m \; \textnormal{and} \; [i_n \ldots i_1 \i_0] \subset [\j]\},$$
so that $\W_{\delta,\i_0}(\i) \in \mathbb{N} \cup \{+\infty\}$.  If $\W_{\delta, \i_0}(\i)=n \in \mathbb{N}$ then by definition $\i_0, i_1 \i_0, \ldots, i_n \cdots i_1 \i_0$ have visited all cylinders in $\{[\j]: \j \in P_\delta\}$. By Proposition \ref{equiv}(i) it follows that the set $\{v_0, S_{i_1}v_0, \ldots, S_{i_n} \cdots S_{i_1}v_0\}$ is $2\delta\Delta$-dense in $F$, so we have $W_{2\delta\Delta,v_0}(\i) \leq n=\W_{\delta, \i_0}(\i)$. Similarly, we obtain $\W_{\delta, \i_0}(\i) \leq W_{\kappa \delta, v_0}(\i)$ from Proposition \ref{equiv}(ii). Hence
\begin{eqnarray}
W_{2\Delta\delta, v_0}(\i) \leq \W_{\delta,\i_0}(\i) \leq W_{\kappa \delta, v_0}(\i). \label{equiv eqn}
\end{eqnarray}
By the definition of $(X_n^\delta)_{n=0}^\infty$ it is clear that $\E_{\i_0}\tau_{\textnormal{cov}}=\E  \W_{\delta,\i_0}$, therefore it follows that Proposition \ref{main2} holds for $(X_n^\delta)_{n=0}^\infty$ where the uniform constant $c \in (0,1)$ can be taken to be $c=\min\{\frac{1}{2\Delta}, \kappa\}$. 

 \end{proof}

\vspace{2mm}

\section{Bounds on the covering time}

Fix $0<\delta<r_{\min}$. In order to prove Theorem \ref{natural} it suffices, via Proposition \ref{main2}, to obtain estimates on the expected covering time $\E_{\i_0} \tau_{\textnormal{cov}}$ for all $\i_0 \in P_\delta$. To this end, we will make extensive use of a family of bounds on the expected covering time of irreducible Markov chains derived from the so called ``Matthews method'' \cite{LePe17,Ma88}. Roughly speaking, this type of bound reduces the task of estimating the expected covering time to only having to estimate the expected time for the Markov chain to travel between a pair of states. In particular, for $\i \in P_\delta$ we define the hitting time 
$$\tau_\i:=\min\{t \geq 0: X^\delta_t=\i\},$$
that is, the first time that the state $\i$ is visited by the Markov chain, and denote by $\E_\i \tau_\j$ the expected hitting time of $\j \in P_\delta$ given that $X^\delta_0=\i$. Then, provided the expected hitting time between a pair of states is fairly homogeneous over a (large part of) the state space, the Matthews method exploits the fact that states can be visited by the Markov chain in many different orders to deduce that the expected covering time is approximately proportional to the logarithm of the cardinality of the state space times the typical expected hitting time between a pair of states. Translated through the symbolic coding introduced in the previous section, the requirement that the expected hitting time be fairly homogenous over a sufficiently large part of the state space turns out to correspond to the situation that the probability vector $(p_1,\ldots,p_N)$ is chosen in such a way that (\ref{t thing}) is not satisfied for a unique index. In this situation the measures $m=\sum_{i=1}^N p_i (S_i)_*m$ of many of the similarly sized pieces $\{S_\i F\}_{\i \in P_\delta}$ are approximately comparable, which implies that expected hitting times between many pairs of states in $P_\delta$ are also roughly uniform. On the other hand, when the expected hitting time between pairs of states varies substantially according to the pair of states which are chosen --  which is the case when (\ref{t thing}) is satisfied uniquely and therefore the measure $m$ is far from being uniformly distributed over any large subcollection of  pieces $\{S_\i F\}_{\i \in P_\delta}$ -- Matthews' method yields bounds which are less sharp, which accounts for the gap between the upper and lower estimates for the rate of growth of $\E W_{\delta, v_0}$ in (\ref{other eqn2}). We will provide precise statements for the ``Matthews method'' bounds which we use in Propositions \ref{matthewsu}, \ref{matthewsu_A} and \ref{matthewsl2}, each of which  will appear directly preceding the proof of the corresponding bound from the part of Theorem \ref{natural} to which it pertains.

 There are a couple more notions related to the covering and hitting times which will be useful in our analysis. Firstly, for $\i \in P_\delta$ we define
$$\tau_\i^+:=\min\{t \geq 1: X^\delta_t=\i\},$$
which we call the first return time to the state $\i$. 


In order to introduce the second one,  it is helpful to visualise the Markov chain $(X_n^\delta)_{n=1}^\infty$ in the following way. At each transition we append a new bit $i \in \{1, \ldots, N\}$ with probability $p_i$ on the left of the current state, say $\i=i_1 \ldots i_n$. There is now a unique way that we can delete the tail of $i\i=ii_1 \ldots i_n$ to yield a new word $\j=ii_1 \ldots i_{n-m} \in P_{\delta}$. Then, $\j$ is the new current state of the chain. In this sense, our Markov chain $(X_n^\delta)_{n=0}^\infty$ is closely related to the Markov chain which describes observing patterns of heads and tails in coin tossing, and we can exploit this connection to adapt techniques which were used in \cite[\S 11.3.3]{LePe17} to compute waiting times for all patterns of a fixed length when tossing a fair coin.

Let $w_{\i}$ denote the first time that $\i$ appears using all new bits, that is, with no overlap with the initial state. This random variable is easier to study than the waiting time $\tau_{\i}$ since it does not depend on the initial state.  There are two trivial observations which will be useful: (i) $w_{\i} \geq \tau_{\i}$ for all $\i \in P_{\delta}$ and (ii) since $w_{\i}$ does not depend on the initial state, $\E w_{\i} \geq \E_{\i}\tau_{\i}^+$, where we have purposefully suppressed the dependence on the initial state in $\E w_{\i}$.

Finally, we fix some extra notation which will be used throughout the proofs. We will write $A \lesssim B$ if $A \leq cB$ for some constant $c$ which depends only on the parameters fixed by the hypothesis of the result being proved, which in our case may be the IFS itself and the probability vector $(p_1,\ldots,p_N)$, but crucially it will never depend on $\delta$, $v_0$ or $\i_0$. Similarly we write $A \gtrsim B$ if $B \lesssim A$ and $A \approx B$ if both $A \lesssim B$ and $B \lesssim A$.

 Also recall that in the proof of Proposition \ref{chain} we introduced the notation
$$L_{\delta}= \max\{|\i|: \i \in P_{\delta}\}=\left\lceil \frac{\log \delta}{\log r_{\min}} \right\rceil.$$
Additionally we will denote
$$\ell_\delta= \min\{|\i|: \i \in P_{\delta}\}=\left\lceil \frac{\log \delta}{\log r_{\max}} \right\rceil.$$

\subsection{Proofs of upper bounds} \hfill \\

First, we obtain the upper bound on $\E W_{\delta,v_0}$ in the case where (\ref{t thing}) is not maximised uniquely in $\{1, \ldots, N\}$, that is, we settle the upper bound in (\ref{other eqn2}). For this we will appeal to Matthews' original upper bound on the expected covering  time of an irreducible Markov chain (see \cite[Theorem 11.2]{LePe17} based on \cite{Ma88}).

\begin{prop}\label{matthewsu}
Fix $0<\delta<r_{\min}$. Then
$$\max_{\i\in P_\delta} \E_{\i} \tau_{\textnormal{cov}} \leq \max_{\i,\j\in P_\delta} \E_\i \tau_\j \left(1+\frac{1}{2}+ \cdots + \frac{1}{|P_\delta|}\right).$$
\end{prop}

We will require the following short lemma which describes which states are most difficult to hit and, given such a state, provides an approximate formula for its stationary probability.

\begin{lma}
 Let $i_0 \in \{1, \ldots, N\}$ satisfy $\frac{\log p_{i_0}}{\log r_{i_0}}= \max_i \frac{\log p_i}{\log r_i}=t$. Given $0<\delta<r_{\min}$ let $\i_0 \in P_\delta$ denote the unique string which is made up only of the digit $i_0$. Then
\begin{eqnarray}
\min_{\i \in P_\delta} p_\i \approx p_{\i_0} \approx \delta^t \label{minp}
\end{eqnarray}
where the implied constants are independent of $\delta$.
\end{lma}

\begin{proof}
Since $t=\frac{\log p_{i_0}}{\log r_{i_0}}$, 
\begin{eqnarray}
p_{\i_0}=r_{i_0}^{t|\i_0|}=r_{\i_0}^t \approx \delta^t, \label{minp1}
\end{eqnarray}
 where the last approximate equality follows because $\delta r_{\min} <r_{i_0}^{|\i_0|-1}r_{\min}\leq r_{\i_0} \leq \delta$ by definition of $P_\delta$.

Also, since $t=\max_{i} \frac{\log p_i}{\log r_i}$, for any $\i=i_1 \ldots i_n \in P_\delta$ we have
\begin{eqnarray}
p_\i=p_{i_1} \cdots p_{i_n}=r_{i_1}^{\frac{\log p_{i_1}}{\log r_{i_1}}} \cdots r_{i_n}^{\frac{\log p_{i_n}}{\log r_{i_n}}} \geq r_\i^t > (\delta r_{\min})^t \approx \delta^t\label{minp2}
\end{eqnarray}
where the second inequality follows because $r_\i \geq r_{i_1} \cdots r_{i_{n-1}} r_{\min} >\delta r_{\min}$ by definition of $P_\delta$. Combining \eqref{minp1} and \eqref{minp2} gives \eqref{minp}.
\end{proof}

\vspace{2mm}

\noindent \textit{Proof of upper bound in (\ref{other eqn2}).} Let $i_0 \in \{1, \ldots, N\}$ satisfy $\frac{\log p_{i_0}}{\log r_{i_0}}= \max_i \frac{\log p_i}{\log r_i}=t$. Fix arbitary $0<\epsilon<r_{\min}$ throughout the proof, and for each $k \in \N$ consider $P_{\epsilon^k}$. Let $\i \in P_{\epsilon^k}$. We can write $\i$ uniquely as $\i_k\i_{k-1} \ldots \i_1$ where for each $1 \leq m \leq k$, $\i_m \ldots \i_1 \in P_{\epsilon^m}$. We note that since $\epsilon<r_{\min}$, $|\i_1| \geq 2$ by definition of $P_\epsilon$. Moreover, for each $1 \leq m \leq k$,
$$r_{\i_m \ldots \i_1} \leq \epsilon^m<r_{\i_m \ldots \i_n \i_1'}$$
where $\i_1'$ denotes $\i_1$ with its last digit removed. Therefore
$$\epsilon^{m+1}<r_{\i_m \ldots \i_1'} \epsilon<r_{\i_m \ldots \i_1'} r_{\min} \leq r_{\i_m \ldots \i_1} \leq \epsilon^m$$
hence $\i_{m+1}$ must necessarily have non-zero length. Let $C_1:= \max_{\j \in P_\epsilon} \E w_{\j} < \infty$ and let $C_2:=\max_{\j \in P_\epsilon} |\j|$. Let $C=\max\{C_1,C_2\}$. 

Now, for each $2 \leq m \leq k$,
\begin{eqnarray}
\E w_{\i_m \ldots \i_1} &\leq& (\E w_{\i_{m-1} \ldots \i_{1}}+C)p_{\i_m}+(\E w_{\i_{m-1} \ldots \i_{1}}+C+\E w_{\i_m \ldots \i_1})(1-p_{\i_m}) \nonumber\\
&=& \E w_{\i_{m-1} \ldots \i_{1}} +C+(1-p_{\i_m}) \E w_{\i_m \ldots p_{\i_1}}. \label{recur1}
\end{eqnarray}
Therefore
\begin{equation}
p_{\i_m}\E w_{\i_m \ldots \i_1} \leq \E w_{\i_{m-1} \ldots \i_{1}}+C .\label{recur2}
\end{equation}
In particular, for $k \geq 2$,
\begin{eqnarray*}
p_{\i_2} \cdots p_{\i_k}\E w_{\i_k \ldots \i_1} &\leq& 2C+Cp_{\i_2}+Cp_{\i_3\i_2}+ \cdots + Cp_{\i_{k-1} \ldots \i_{2}} \\
&\leq& C+\frac{C}{1-p},
\end{eqnarray*}
where $p:= \max_{i \in \I} p_i$.  Therefore, 
\begin{eqnarray}
\E w_{\i} \lesssim p_{\i}^{-1}
\label{key ub}
\end{eqnarray}
where the implied constant depends on the probability vector $(p_1,\ldots,p_N)$ and on the choice of $\epsilon$ but does not depend on the length $k=|\i|$.

Now, fix $\delta>0$, and let $k \in \N$ satisfy $\epsilon^k \leq \delta <\epsilon^{k-1}$. Let $\j \in P_\delta$. Then we can choose some $\i \in P_{\epsilon^k}$ such that $\i|_n = \j$ for some $n \in \N$. Then $p_\i \geq p_{\j} p_{min}^{C_2} \approx p_{\j}$ where $p_{min}:=\min_{i}p_i$.  In particular
\begin{eqnarray}
\E w_{\j} \leq \E w_\i \lesssim p_\i^{-1} \lesssim p_{\j}^{-1}.
\label{key ub2}
\end{eqnarray}
Let $\i_0 \in P_\delta$ be the unique string consisting only of the digit $i_0$. By \eqref{minp},
\begin{eqnarray}
\max_{\j \in P_\delta}\E w_{\j} \lesssim p_{\j}^{-1} \lesssim p_{\i_0}^{-1}\approx \delta^{-t}.
\label{key ub3}
\end{eqnarray}

 By  Proposition \ref{matthewsu},
\begin{eqnarray*}
\E_{\i_0} \tau_{\textnormal{cov}} &\leq& \left(1+\frac{1}{2}+ \cdots + \frac{1}{|P_{\delta}|}\right)\max_{\i, \j \in P_{\delta}}\E_\i \tau_\j   \\
&\leq&   \left(1+\frac{1}{2}+ \cdots + \frac{1}{N_\delta}\right)\max_{ \j \in P_{\delta}}\E w_\j  \\
&\lesssim& \delta^{-t}\log\left(\frac{1}{\delta}\right)
\end{eqnarray*}
where we have used the fact that $N_\delta := |P_\delta| \leq r_{\min}^{-s}\delta^{-s}$ as remarked at the beginning of the previous section. The result now follows by Proposition \ref{main2}. 
\qed

\vspace{2mm}

In the above proof we did not refer to the question of whether or not \eqref{t thing} is uniquely maximised only at $i_0$, and indeed the upper bound of $ \E W_{\delta, v_0} \lesssim \delta^{-t} \log\left(\frac{1}{\delta}\right)$ for all $v_0 \in F$ is valid whether or not this is the case. However, in the case where (\ref{t thing}) is maximised uniquely this upper bound is not optimal owing to the fact that on a large part of the state space $P_\delta$, the expected hitting times are significantly lower than $\delta^{-t}$. Indeed for a typical choice of $\i \in P_\delta$ an expected hitting time $\E_\i \tau_\j$ will be of the order $\delta^{-t}$ only if the word $\j$ contains many instances of the digit $i_0$.  Thus in the case where \eqref{t thing} is maximised at a unique index it is useful to separate the less accessible part of the state space from the more accessible part of the state space and estimate the expected covering times of each of these parts separately. 

To be precise, for any $B \subset P_\delta$ we define
$$\tau_{\textnormal{cov}}^B= \min\{t \geq 0: \, \forall \i \in B, \; \exists s \leq t \; \textnormal{s.t.} \; X_s=\i\},$$
that is, the first time that all of the states in $B$ have been visited by the chain. The expected covering time of $B$ satisfies an analogous upper bound to the expected covering time (see \cite[(11.16)]{LePe17} which can easily be proven by adapting the proof of \cite[Theorem 11.2]{LePe17}):

\begin{prop}\label{matthewsu_A}
Fix $0<\delta<r_{\min}$ and $B \subset P_\delta$. Then
$$\max_{\i\in B} \E_\i \tau_{\textnormal{cov}}^B \leq \max_{\i,\j \in B} \E_\i \tau_\j \left(1+\frac{1}{2}+ \cdots + \frac{1}{|B|}\right).$$
\end{prop}

In particular, in the case where \eqref{t thing} is uniquely maximised one can improve on the upper bound of $ \E W_{\delta, v_0} \lesssim \delta^{-t} \log\left(\frac{1}{\delta}\right)$ by instead estimating how long it would take for the Markov chain to first visit all states in a subset $B$ consisting of states $\i \in P_\delta$ which contain a restricted number of the digit $i_0$, followed by visiting all states in $B'=P_\delta \setminus B$. The subset $B$ would constitute most of the state space $P_\delta$, but would benefit from having a reduced upper bound on $\max_{\i,\j \in B} \E_\i \tau_\j$. On the other hand, the upper bound on  $\max_{\i,\j \in B'} \E_\i \tau_\j$ would be of the order $\delta^{-t}$, but $B'$ would only comprise a small proportion of the state space. Of course, there is a lot of flexibility in how the subset $B$ could be defined, so by carefully considering the contribution of each covering time $\E_\i\tau_{\textnormal{cov}}^B$ and $\E_\i\tau_{\textnormal{cov}}^{B'}$ one can choose $B$ in such a way that the upper bound on $ \E W_{\delta, v_0} $ is improved to the degree noted in Theorem \ref{natural}:

\vspace{2mm}

\begin{proof}[Proof of upper bound in \eqref{other eqn1}]
Suppose that $i_0 \in \{1, \ldots, N\}$ is the unique digit that satisfies $\frac{\log p_{i_0}}{\log r_{i_0}}= \max_i \frac{\log p_i}{\log r_i}=t$. Fix $\delta>0$ and consider $k_\delta$ satisfying $1\leq k_{\delta}\leq\ell_\delta$ to be chosen later. Define $B \subset P_\delta$ to be all strings which contain at most $k_\delta-1$ digits from the set $\{1, \ldots, N\} \setminus \{i_0\}$. Also denote $B'= P_\delta \setminus B$, so that $B'$ are all strings in $P_\delta$ which contain at least $k_\delta$ digits from the set $\{1, \ldots, N\} \setminus \{i_0\}$. 

We can bound the covering time above by the time it would take to first cover $B$ and then cover $B'$ (plus the intermediate travel time) yielding
\begin{eqnarray}
\max_{\i \in P_\delta} \E_\i \tau_{\textnormal{cov}} \leq \max_{\i \in P_\delta} \min_{\j \in B} \E_\i \tau_{\j} + \max_{\i \in B} \E_\i \tau_{\textnormal{cov}}^B + \max_{\i \in P_{\delta}} \min_{\j \in B'} \E_\i \tau_\j + \max_{\i \in B'} \E_\i \tau_{\textnormal{cov}}^{B'}.
\label{ub key}
\end{eqnarray}
We will use Proposition \ref{matthewsu_A} to bound $ \max_{\i \in B} \E_\i \tau_{\textnormal{cov}}^B$ and $\max_{\i \in B'} \E_\i \tau_{\textnormal{cov}}^{B'}$ in terms of $k_\delta$, before choosing $k_\delta$ in a way that optimises this bound.

To calculate $|B|$, notice that $B= \bigcup_{m=0}^{k_\delta-1} B_m$ where $B_m$ is the set of strings in $P_\delta$ which contain exactly $m$ occurrences of digits from the set $\{1, \ldots, N\} \setminus \{i_0\}$. Since any $\i \in B$ has length $|\i| \leq L_\delta$, there are no more than $L_\delta$ positions where the first digit from the set $\{1, \ldots, N\} \setminus \{i_0\}$ can appear, followed by no more than $L_\delta-1$ positions where the second digit from the set $\{1, \ldots, N\} \setminus \{i_0\}$ can appear and so on. Since we have $N-1$ possible choices of digits for each of these,
\begin{eqnarray*}
|B_m| &\leq& (N-1)^m L_\delta (L_\delta-1) \cdots (L_\delta -m+1) \\
&\leq& (N-1)^mL_\delta^m.
\end{eqnarray*}
Therefore,
$$|B| \leq \sum_{m=0}^{k_\delta-1} (N-1)^m L_\delta^m \lesssim ((N-1)L_\delta)^{k_\delta}.$$
Since $L_\delta \lesssim \log \left(\frac{1}{\delta}\right)$ we have $\log|B| \lesssim k_\delta \log \log \left(\frac{1}{\delta}\right)$. Notice that $B$ contains a string $\i_0$ which contains only the digit $i_0$, therefore $\min_{\i \in B} p_\i \approx p_{\i_0} \approx \delta^t$ by \eqref{minp}.

Next we consider $B'$. Since $B'=P_\delta \setminus B$ we have $\log |B'| \approx \log \left(\frac{1}{\delta}\right)$. Next we calculate $\min_{\i \in B'} p_\i$. Let $\i \in B'$. Then for some integer $m$ satisfying $k_\delta \leq m \leq L_\delta$, $\i$ contains $m$ digits from the set $\{1, \ldots, N\} \setminus \{i_0\}$. In particular there exist $j_1, \ldots , j_m \in \I \setminus \{i_0\}$ and some integer $n$ such that
$$r_\i=r_{j_1} \cdots r_{j_m} r_{i_0}^n \approx \delta.$$
We have
$$p_{i_0}^n=r_{i_0}^{n \frac{\log p_{i_0}}{\log r_{i_0}}} \approx \left(\frac{\delta}{r_{j_1} \cdots r_{j_m}}\right)^t.$$
Therefore
$$p_\i\approx \delta^t \frac{p_{j_1} \cdots p_{j_m}}{r_{j_1}^t \cdots r_{j_m}^t}.$$
Fix
$$c= t- \max_{i \neq i_0} \frac{\log p_i}{\log r_i} >0,$$
where positivity follows from the fact that $i_0$ uniquely achieves the maximum in (\ref{t thing}). It follows that for each $i \in \{1, \ldots, N\} \setminus \{ i_0\}$,
$$\frac{p_i}{r_i^t}=\frac{r_i^{\frac{\log p_i}{\log r_i}}}{r_i^t}= r_i^{\frac{\log p_i}{\log r_i}-t} \geq r_i^{-c}.$$
Therefore
$$p_\i \gtrsim \delta^t r_i^{-cm} \geq \delta^t r_{\max}^{-ck_\delta}.$$
Note that by (\ref{key ub2}),
$$\max_{\i \in P_\delta} \min_{\j \in B} \E_\i \tau_\j \leq \min_{\j \in B} \E w_\j \lesssim \min_{\j \in B} p_\j^{-1} \lesssim \delta^{-t}$$
and
$$\max_{\i \in P_\delta} \min_{\j \in B'} \E_\i \tau_\j \leq \min_{\j \in B'} \E w_\j \lesssim \min_{\j \in B'} p_\j^{-1} \lesssim \delta^{-t}r_{\max}^{ck_\delta}.$$
Therefore, by Proposition \ref{matthewsu_A} and (\ref{ub key}),
\begin{eqnarray}
\max_{\i \in P_\delta} \E_\i\tau_{\textnormal{cov}} &\lesssim& \delta^{-t} k_\delta \log \log \left(\frac{1}{\delta}\right) + \delta^{-t} r_{\max}^{ck_\delta} \log \left(\frac{1}{\delta}\right) \label{optimise}.
\end{eqnarray}
Consider the function $f_\delta(x)= x \log \log \left(\frac{1}{\delta}\right) -r_{\max}^{cx} \log \left(\frac{1}{\delta}\right)$. Observe that $f_\delta(2)<0$ and $f_\delta(\ell_\delta)>0$, provided $\delta$ is sufficiently small. Therefore there exists $2<x_\delta<\ell_\delta$ such that $f_\delta(x_\delta)=0$. Notice that $f_\delta(\log \log \left(\frac{1}{\delta}\right))<0$ provided $\delta$ is sufficiently small. Since $f_\delta(x)$ is increasing with $x$, $x_\delta> \log \log \left(\frac{1}{\delta}\right)$. Therefore, if we fix $k_\delta=\lfloor x_\delta \rfloor$, we have 
$$k_\delta \log \log \frac{1}{\delta} \approx r_{\max}^{ck_\delta} \log \left(\frac{1}{\delta}\right)$$
and $k_\delta \lesssim \log \log \left(\frac{1}{\delta}\right)$. 
Hence by (\ref{optimise}),
\begin{eqnarray*}
 \max_{\i_0 \in P_\delta} \E_{\i_0}\tau_{\textnormal{cov}} &\lesssim & \delta^{-t} k_\delta \log \log \left(\frac{1}{\delta}\right) \\
&\leq& \delta^{-t}  \left(\log \log \left(\frac{1}{\delta}\right)\right)^2.
\end{eqnarray*}
The result now follows from Proposition \ref{main2}.
\end{proof}

\vspace{2mm}

\subsection{Proofs of lower bounds} \hfill \\

To estimate the lower bounds on the expected covering time, we begin by obtaining a lower estimate for the expected hitting time $\E_\i \tau_\j$ in terms of $\E w_\j$.

\begin{lma} \label{theta}
Given arbitrary $\i, \j \in P_{\delta}$ where $\i \neq \j$ denote
$$\theta_{\i, \j}= \P_{\i}(\tau_{\j}<L_{\delta}).$$
Then
\begin{eqnarray}
 \E_{\i} \tau_{\j} \geq \E w_{\j} -\theta_{\i,\j}(L_{\delta}+\E w_{\j}). \label{theta ineq}
\end{eqnarray}
\end{lma}

\begin{proof}
Observe that
\begin{eqnarray}
w_{\j} \leq \tau_{\j} + \mathbbm{1}_{\{\tau_{\j}<L_{\delta}\}}(L_{\delta}+w_{\j}^*) \label{pre theta}
\end{eqnarray}
where $w_{\j}^*$ is the amount of time required to build $\j$ from new bits after the $L_{\delta}$th bit has been added. Indeed (\ref{pre theta}) holds since if $(X_t)$ is such that $\j$ appears for the first time after $L_{\delta}$ time then $w_{\j}((X_t))=\tau_{\j}((X_t))$ whereas if $(X_t)$ is such that $\j$ appears for the first time before $L_{\delta}$ time then $w_{\j}((X_t)) \leq L_{\delta}+w_{\j}^*((X_t))$. 

Now, since $w_{\j}^*$ is independent of the event $\{\tau_{\j}<L_{\delta}\}$ and $w_{\j}$ has the same distribution as $w_{\j}^*$, we can take expectations in (\ref{pre theta}) to obtain
$$\E w_{\j} \leq \E_{\i} \tau_{\j} +\theta_{\i,\j}(L_{\delta}+\E w_{\j})$$
which completes the proof of (\ref{theta ineq}).
\end{proof}

The usefulness of (\ref{theta ineq}) is that $\E w_\j \geq \E_\j \tau_\j^+$, and the expected return time $\E_\j \tau_\j^+$ satisfies  the following formula (see \cite[ Proposition 1.14 ]{LePe17}).

\begin{prop} \label{return}
Fix $0<\delta<r_{\min}$. For all $\i\in P_\delta$,
$$\E_{\i}(\tau_\i^+)= \frac{1}{\pi_\i}= \frac{1}{p_\i}.$$
\end{prop}

So, in order to apply (\ref{theta ineq}) we require an estimate on the probability $\theta_{\i,\j}$ of a fast hitting time of state $\j$ from state $\i$, which is provided by the following lemma.

\begin{lma} \label{trans}
Fix $\i, \j \in P_{\delta}$, where $\i \neq \j$. Suppose that $\tau_{\j}((X_t)) \geq j$ whenever $X_0=\i$. Then
$$\theta_{\i, \j}= \P_{\i}(\tau_{\j}<L_{\delta}) \leq \frac{p^j}{1-p} +p^{\ell_\delta}(L_{\delta}-\ell_\delta).$$
\end{lma}

\begin{proof} 
Fix such $\i$ and $\j$. We have
$$ \P_{\i}(\tau_{\j}<L_{\delta}) \leq \P_{\i}(\tau_{\j}=j)+ \P_{\i}(\tau_{\j}=j+1)+ \cdots+ \P_{\i}(\tau_{\j}=L_{\delta}-1). $$
For each $j \leq k \leq |\j|-1$, $\P_\i(\tau_\j=k) \leq p^k$, since $k$ correct transitions are required (which correspond to the $k$ correct digits that need to be appended to the left of the word $\i$). If $|\j|< L_{\delta}$, for each $|\j| \leq k \leq L_{\delta}-1$, $\P_\i(\tau_\j=k) \leq p^{|\j|}$, since $|\j|$ correct transitions are required. Therefore
\begin{eqnarray*}
 \P_{\i}(\tau_{\j}<L_{\delta}) &\leq& \P_{\i}(\tau_{\j}=j)+ \P_{\i}(\tau_{\j}=j+1)+ \cdots+  \P_{\i}(\tau_{\j}=L_{\delta}-1) \\
&\leq& p^{j}+p^{j+1}+ \cdots + p^{|\j|-1}+(L_{\delta}-|\j|)p^{|\j|} \\
&\leq &p^{j}+p^{j+1}+ \cdots + p^{\ell_\delta-1}+(L_{\delta}-\ell_\delta)p^{\ell_\delta}\\
&=& p^{j} \frac{1-p^{\ell_\delta-j-2}}{1-p} +(L_{\delta}-\ell_\delta)p^{\ell_\delta}\\
&\leq& \frac{p^j}{1-p} +p^{\ell_\delta}(L_{\delta}-\ell_\delta).
\end{eqnarray*}
\end{proof}

We are almost ready to prove the lower bounds on $\E W_{\delta,v_0}$ from (\ref{other eqn1}) and (\ref{other eqn2}) via appropriate lower bounds on the expected covering time of $(X_n^\delta)_{n=0}^\infty$. Analogously to Proposition \ref{matthewsu}, the original lower bound of Matthews \cite{Ma88} bounds the minimum expected covering time from below by the minimum expected hitting times between distinct states multiplied by the logarithm of the cardinality of the state space. Clearly, this bound is insufficient for our purposes, since it can yield a lower bound merely of the order $\log\left(\frac{1}{\delta}\right)$ owing to the fact in general some states in $P_\delta$ will be extremely close to one other. Instead, we can again improve on this bound by considering the expected covering time of a subset of the state space, where this time the elements in the subset are chosen in such a way that they are all ``uniformly far'' from each other in the sense that the expected hitting times of possible pairs of states are uniformly bounded below. For this we will require the following analogue of Proposition \ref{matthewsu_A} (see \cite[Proposition 11.4 ]{LePe17}\footnote{Although Proposition 11.4 in \cite{LePe17} is not stated exactly as it is here, Proposition \ref{matthewsl2} can easily be gleaned from the proof of \cite[Proposition 11.4]{LePe17}.}).

\begin{prop}\label{matthewsl2}
Let $0<\delta<r_{\min}$ and $B \subset P_\delta$. Then for all $\i_0 \in B$,
$$\E_{\i_0} \tau_{\textnormal{cov}}\geq \E_{\i_0} \tau_{\textnormal{cov}}^B \geq \min_{\i,\j \in B, \i \neq \j} \E_\i(\tau_\j) \left(1+\frac{1}{2}+ \cdots + \frac{1}{|B|-1}\right).$$
\end{prop}

We are now ready to obtain a lower bound on $\E W_{\delta,v_0}$ in the case that (\ref{t thing}) is maximised uniquely at some $i_0 \in \{1, \ldots, N\}$. Since the least accessible part of the state space $P_\delta$ comprises states $\i \in P_\delta$ which consist mostly of the digit $i_0$, the most effective choice of $B$ in Proposition \ref{matthewsl2} is a subset of this type. Then applying the lower estimates from Lemma \ref{theta} on the expected hitting time will yield the desired result.

\vspace{2mm}

\begin{proof}[Proof of lower bound in (\ref{other eqn1})]
Fix $\delta>0$ sufficiently small such that $\frac{p^{\ell_\delta-1}}{1-p}+p^{\ell_\delta}(L_\delta-\ell_\delta) \leq \frac{1}{2}$ and denote $i_0 \in \{1, \ldots, N\}$ to be the digit that satisfies $\frac{\log p_{i_0}}{\log r_{i_0}}= \max_i \frac{\log p_i}{\log r_i}=t$. Let $j \in \{1, \ldots, N\} \setminus \{i_0\}$. Define
$$A=\{\i \in P_{\delta/r_j^2}: \textnormal{$\i$ contains one instance of $j$ and $|\i|-1$ instances of $i_0$}\}$$
and $B= \{jj \i: \i \in A\} \subset P_\delta$. For any $\j \in B$, $\pi_\j =p_\j \approx \delta^{t}$ and therefore for any $\i, \j \in B$, where $\i \neq \j$ we have
\begin{eqnarray*}
\E_\i \tau_\j &\geq&  (1-\theta_{\i,\j}) \E w_\j -L_{\delta}\theta_{\i,\j}\\
&\geq& (1-\theta_{\i,\j}) \E_\j \tau_\j^+ -L_{\delta}\theta_{\i,\j} \\
&\gtrsim& (1-\theta_{\i,\j}) \delta^{-t} -L_{\delta} \theta_{\i,\j}
\end{eqnarray*}
by (\ref{theta ineq}) and Proposition \ref{return}.
In order to bound $\theta_{\i,\j}$, notice that by definition of $B$, at least $|\j|-1$ transitions are required to hit the state $\j$ from $\i$. Thus by Lemma \ref{trans} and our assumption on $\delta$,
$$\theta_{\i, \j}\leq \frac{p^{|\j|-1}}{1-p} +p^{\ell_\delta}(L_{\delta}-\ell_\delta) \leq \frac{p^{\ell_\delta-1}}{1-p} +p^{\ell_\delta}(L_{\delta}-\ell_\delta) \leq \frac{1}{2}.$$



Finally, to calculate $|B|$, observe that there are at least $l_{\delta/r_j^2}\approx \log \frac{1}{\delta}$ distinct positions at which the digit $j$ can be placed within a string $\i \in A$ and therefore $\log|B|=\log|A| \approx \log\log \left(\frac{1}{\delta}\right)$.

 By Proposition \ref{matthewsl2}
\begin{eqnarray*}
 \min_{\i_0 \in P_\delta}\E_{\i_0} \tau_{\textnormal{cov}} &\geq& \min_{\i\in P_{\delta}}\min_{\j \in B}\E_\i \tau_\j + \min_{\i \in B} \E_{\i}  \tau^B_{\textnormal{cov}} \\
&\geq& \left(1+\frac{1}{2}+ \cdots + \frac{1}{|B|-1}\right) \min_{\substack{\i, \j \in B \\ \i \neq \j}} \E_{\i} \tau_{\j}\\
& \gtrsim& \delta^{-t}\log\log \left(\frac{1}{\delta}\right).\end{eqnarray*}
The result follows by Proposition \ref{main2}. \end{proof}

\vspace{2mm}

All that remains is for us to  obtain the lower bound in Theorem \ref{natural} in the case that (\ref{t thing}) is not uniquely maximised in $\{1, \ldots, N\}$. Since there must be at least two digits which attain this maximum, Proposition \ref{matthewsl2} can be applied for a choice of $B \subset P_\delta$ where the cardinality of $B$ is exponential in $\delta^{-1}$. This allows us to recover a sharp lower bound for $\E W_{\delta, v_0}$ in this case.

\vspace{2mm}

\begin{proof}[Proof of lower bound in (\ref{other eqn2}).]
Let $\mathcal{J} \subset \{1, \ldots, N\}$ be the set for which the maximum in (\ref{t thing}) is attained, where $|\mathcal{J}| \geq 2$ by assumption. Define
$$P_{\delta}'=P_{\delta} \cap \mathcal{J}^*$$
where $\mathcal{J}^*$ denotes the set of finite words with digits in $\mathcal{J}$. We begin by showing that there exists $a>0$ such that $|P_{\delta}'| \approx \delta^{-a}$. Let $a >0$ satisfy $\sum_{i \in \mathcal{J}}r_i^a=1$ and let $\P$ be the Bernoulli measure on $\Sigma$ where $\P([i])=r_i^a$ if $i \in \mathcal{J}$ and $\P[i]=0$ otherwise. Then $1=\sum_{\i \in P_{\delta}} p_{\i} = \sum_{\i \in P_{\delta}'} r_{\i}^a$, therefore $r_{\min}^a\delta^a|P_{\delta}'| < 1 \leq \delta^a|P_{\delta}'|$, implying that $|P_{\delta}'| \approx \delta^{-a}$. 

Fix $j \in \N$ sufficiently large such that $\frac{p^j}{1-p} \leq \frac{1}{4}$. Fix $\delta>0$ sufficiently small such that $\ell_\delta >j$ and $p^{\ell_\delta}(L_\delta-\ell_\delta) \leq \frac{1}{4}$.  Fix arbitrary $\i_0 \in P_{\delta}$ and denote the first $j$ digits of $\i_0$ by $i_1 \ldots i_{j}$.  Define a new word $\k=k_1 \ldots k_{j}$ by setting $k_1=i_1$ and fixing $k_2= \cdots= k_j=w_1 \in \mathcal{J} \setminus \{i_1\}$. Note that it is not necessarily true that $\{i_1\} \subset \mathcal{J}$, but if this is the case then $w_1 \in \mathcal{J} \setminus \{i_1\}$ can always be chosen since $|\mathcal{J}| \geq 2$. Define $B_{\i_0} \subset P_{\delta}$ as the set 
$$B_{\i_0}=  \{\i \in P_{\delta}: \i=\k \j \; \textnormal{for some} \; \j \in \mathcal{J}^*\}.$$
We begin by claiming that for all $\delta>0$ sufficiently small and all $\i \in B_{\i_0}$, $\pi_\i \lesssim \delta^{t}$. Writing $\i=\k\j$ as in the definition of $B_{\i_0}$, we have
$$\pi_\i=p_\i=p_\k p_\j \leq \frac{\delta^t}{r_{\min}^{jt}} p^{j} \lesssim \delta^t. $$ 

Next we estimate $\theta_{\i,\j}=\P_{\i}(\tau_\j^+ < L_{\delta})$ for $\i, \j \in B_{\i_0}$, $\i \neq \j$. Fix such $\i$ and $\j$. Observe that both $\i|_j=\k$ and $\j|_j=\k$. Since $k_1$ does not agree with $k_2, \ldots, k_j$, at least $j$ transitions are required before the chain can hit the state $\j$, when starting from state $\i$. Therefore by Lemma \ref{trans} and our assumptions on $j$ and $\delta$,
$$\theta_{\i,\j} \leq \frac{p^j}{1-p} + (L_{\delta}-\ell_\delta)p^{\ell_\delta}  \leq \frac{1}{2}.$$

Next we bound $|B_{\i_0}|$.  Observe that 
$$B_{\i_0}=\{\k\j: \j \in P_{\frac{\delta}{r_{\k}}}'\},$$ 
therefore $|B_{\i_0}|=|P_{\frac{\delta}{r_{\k}}}'|\approx \delta^{-a}$, where $a$ satisfies $\sum_{i \in \mathcal{J}}r_i^a=1$, as before. Therefore by Lemma \ref{matthewsl2}, for any $\j \in B_{\i_0}$,
$$\E_{\j}(\tau_{\textnormal{cov}}^{B_{\i_0}}) \gtrsim \delta^{-t} \log \left( \frac{1}{\delta}\right).$$
By Proposition \ref{matthewsl2}, 
\begin{eqnarray*}
 \min_{\i_0\in P_{\delta}}\E_{\i_0}  \tau_{\textnormal{cov}} 
&\geq& \min_{\i_0 \in P_{\delta}} \min_{\j \in B_{\i_0}} \E_{\i_0}(\tau_\j) + \min_{\i_0\in P_{\delta}} \min_{\j \in B_{\i_0}} \E_{\j}(\tau_{\textnormal{cov}}^{B_{\i_0}}) \\
& \gtrsim &\delta^{-t}\log \left( \frac{1}{\delta}\right).\end{eqnarray*}
The result follows by Proposition \ref{main2}.
\end{proof}

\section{Directions for future research}

Besides the problem of obtaining a sharp estimate for the asymptotic behaviour of the expected $\delta$-waiting time in the case where the maximum in \eqref{t thing} is attained uniquely, several further directions of research suggest themselves. On the one hand, while this work helps to shed light on how the sequence $(x_n)_{n=0}^\infty$ approaches the attractor set, we have not investigated the related question of how quickly the measures $\frac{1}{n}\sum_{k=0}^{n-1}\delta_{x_k}$ approach the self-similar limit measure $m=\sum_{i=1}^N r_i (S_i)_*m$ (with respect to, for example, the Wasserstein distance) and this question may be of interest in future research. It is also interesting to ask how far these results may be extended to the context of iterated function systems defined by maps which are not similarities (such as affine or conformal differentiable transformations) and to cases where the open set condition is not satisfied. Finally, we note that there are analogous questions which make sense for deterministic chaotic dynamical systems. For example, if $T \colon \mathbb{R}/\mathbb{Z} \to \mathbb{R}/\mathbb{Z}$ is the doubling map $T(x):=2x \mod 1$, then for Lebesgue almost every $x \in \mathbb{R}/\mathbb{Z}$ the sequence $\{x,Tx,T^2x,\ldots,\}$ is dense in $\mathbb{R}/\mathbb{Z}$. One could just as easily ask how the expectation with respect to $x$ of the first integer $n$ such that the sequence $\{x,Tx,\ldots,T^{n-1}x\}$ is $\delta$-dense in $\mathbb{R}/\mathbb{Z}$ behaves as a function of $\delta$ in the limit $\delta \to 0$. For the doubling map $T(x):=2x\mod 1$ this question can be reduced via Markov partitions to the coupon-collector's problem, but for smooth expanding maps or even Anosov diffeomorphisms the details of such an argument are less clear.

\section{Acknowledgements}

 Both authors were financially supported by the \emph{Leverhulme Trust} (Research Project Grant number RPG-2016-194). NJ was also financially supported by the \emph{EPSRC} (Standard Grant EP/R015104/1). NJ thanks John Sylvester for enlightening conversations about covering problems for Markov chains.

\bibliographystyle{acm}
\bibliography{other}
\end{document}